\documentclass{amsproc}
\usepackage{mathptmx}      
\usepackage{latexsym}
\usepackage{longtable}
\usepackage{amsmath, amsthm, amscd, amsfonts, amssymb, graphicx, color}
\usepackage{latexsym, amsmath}
\usepackage{amssymb}

\usepackage{amssymb}
\usepackage{amsfonts}
\usepackage{amsbsy}
\usepackage{latexsym}
\usepackage[colorlinks]{hyperref}
\hypersetup{colorlinks, citecolor=blue}

\newcommand\tr{\operatorname{Tr_g}}
\newcommand\Div{\operatorname{div}}
\def\Ric{\operatorname{Ric_g}}

\newtheorem{theo}{\bf Theorem}[section]
\newtheorem{lem}{\bf Lemma}[section]

\theoremstyle{remark}

\numberwithin{equation}{section}

\begin{document}
	\newcommand*{\vertbar}{\rule[-1ex]{0.5pt}{2.5ex}}
	\newcommand*{\horzbar}{\rule[.5ex]{2.5ex}{0.5pt}}
	
	\title[Geometry of almost $*$-Ricci solitons]{Geometry of almost contact metrics as almost $*$-Ricci solitons}
     \author{Dhriti Sundar Patra}
     \address{Department of Mathematics,University of Haifa, Mount Carmel, 31905 Haifa,  Israel}
     \email{dhriti.ju@gmail.com}
\author{Akram Ali}
     \address{Department of Mathematics, College of Science,\\ King Khalid University, 9004 Abha, Saudi Arabia\\
     Departamento de Mathematica, Instituto de Cincias Exacts-ICE Universidade Federal do Amazonas, Av. General Rodrigo Octavio, 6200, 69080-900 Manaus, Amazonas, Brazil.}
     \email{akramali133@gmail.com; akramali@ufam.edu.br}

     \author{Fatemah Mofarreh}
     \address{Mathematical Science Department Faculty of Science Princess Nourah bint Abdulrahman University. Riyadh 11546 Saudi Arabia.}	
\email{fyalmofarrah@pnu.edu.sa}

	\date{14/05/2017}
	
	\dedicatory{}
	
	\begin{abstract}
	In the present paper, we give some characterizations by considering $*$-Ricci soliton as a Kenmotsu metric. We prove that if a Kenmotsu manifold represents an almost $*$-Ricci soliton with the potential vector field $V$ is a Jacobi along the Reeb vector field, then it is a steady $*$-Ricci soliton. Next, we show that a Kenmotsu matric endowed an almost $*$-Ricci soliton is Einstein metric if it is $\eta$-Einstein or the potential vector field $V$ is collinear to the Reeb vector field or $V$ is an infinitesimal contact transformation.	
	\end{abstract}
	
	\maketitle
	
\section{Introduction}
    In \cite{pigola2011ricci}, Pigola et al. introduced the notion of almost Ricci soliton as a generalization of Ricci soliton by adding the condition on the parameter $\lambda$ to be a variable function. More precisely, a Riemannian manifold $(M,g)$ is said to be a Ricci soliton if there exist
    a vector field $V$ and a constant $\lambda$ satisfying
    \begin{align}\label{1.0}
    \frac{1}{2}\,\pounds_{V}g + \Ric = \lambda\,g,
    \end{align}
    where $\pounds_V$ denotes the Lie-derivative in the direction of $V$ and $\Ric$ is the Ricci tensor of $g$. The study on the Ricci solitons and almost Ricci solitons have a long history and a lot of achievements were acquired (see \cite{fernandez2008remark,munteanu2013gradient,petersen2009rigidity,pigola2011ricci}). On the other hand, the almost contact manifolds admitting Ricci solitons or almost Ricci solitons were also been studied by many researchers, see \cite{Chen1, Chen2,ghosh2013eta,ghosh2020ricci,ghosh2019ricci,P20,sharma2008certain,sharma2014almost,wang2017ricci}.\\
    
    The notion of $*$-Ricci tensor was
    introduced by Tachibana \cite{tachibana1959} on almost Hermitian manifolds and studied by Hamada and Inoguchi \cite{hamada2002real} on real hypersurfaces of non-flat complex space forms(for details see \cite{ivey2010ricci}) and defined by
    \begin{align*}
    	\Ric^{*}(X,Y) = \frac{1}{2}\tr\{\varphi o R(X,\varphi\,Y)\}
    \end{align*}
    for all vector fields $X$, $Y$ on $M$ and $\varphi$ is a $(1,1)$-tensor field. Furthermore, Hamada gave a complete classification of $*$-Einstein Hopf hypersurfaces in non-flat complex space form. If the $*$-Ricci tensor is a constant multiple of the Riemannian metric $g$, then $M$ is said to be a $*$-Einstein manifold. As a generalization of $*$-Einstein metric, Kaimakamis et al. \cite{kaimakamis2014ricci2} introduced a so-called $*$-Ricci soliton, where essentially they modified the definition of Ricci soliton by replacing the Ricci tensor $\Ric$ in soliton equation \eqref{1.0} with the $*$-Ricci tensor $\Ric^{*}$ and defined by
    \begin{align}\label{1.1}
    \frac{1}{2}\,\pounds_{V}g + \Ric^* = \lambda\,g.
    \end{align}
    Note that a $*$-Ricci soliton is trivial if the vector field $V$ is Killing, and in this case, the manifold becomes \textit{$*$-Einstein}.
    Thus, it is considered as a natural generalization of $*$-Einstein metric. A $*$-Ricci soliton is said to be almost $*$-Ricci soliton if $\lambda$ is a smooth function on $M$. Moreover, an almost $*$-Ricci soliton is called shrinking, steady and expanding according to as $\lambda$ is positive, zero and negative, respectively. In \cite{kaimakamis2014ricci2},  it was studied a real hypersurfaces of a non-flat complex space form admitting a $*$-Ricci soliton whose potential vector field is the structure vector field and proved that a real hypersurface in a complex projective space does not admit a $*$-Ricci soliton. They have also shown that a real hypersurface of complex hyperbolic space admitting a $*$-Ricci soltion is locally congruent to a geodesic hypersphere.\\

    %
    %
    %
    
    In \cite{kenmotsu1972class}, Kenmotsu first introduced and studied a particular class of almost contact metric manifolds, known as Kenmotsu manifolds, characterized by the terms of warped products.
    The warped product $\mathbb{R}\times_{f} N$ (of the real line $\mathbb{R}$ and a K\"{a}hler manifold $N$) with the warping function
    \begin{align}\label{E-f-warp}
    f(t)=ce^t,
    \end{align}
    given on an open interval $J=(-\varepsilon, \varepsilon)$,
    admits Kenmotsu structure. Conversely, every point of a Kenmotsu manifold has a neighbourhood, which is locally a warped product $J\times_{f} N$, where $f$ is given by \eqref{E-f-warp}.
    In recent years, several anthers studied Ricci soliton and almost Ricci soliton on contact metric manifolds \cite{sharma2008certain,sharma2014almost}, Kenmotsu manifolds \cite{ghosh2013eta,ghosh2019ricci}, cosymplectic manifold \cite{wang2017ricci}. First, Ghosh \cite{ghosh2019ricci} proved that a Kenmotsu metric as a Ricci soliton is Einstein if the metric is $\eta$-Einstein \cite{ghosh2013eta} or the potential vector field $V$ is contact. Recently, Patra-Rovenski \cite{patra2020almost} proved that a Kenmotsu metric as an $\eta$-Ricci soliton is Einstein if either it is $\eta$-Einstein or the potential vector field $V$ is an infinitesimal contact transformation or $V$ is collinear to the Reeb vector field.
    
     Next, Chen \cite{xiaomin2018real} considered a real hypersurface of a non-flat complex space form which admits a $*$-Ricci soliton whose potential vector field belongs to the principal curvature space and the holomorphic distribution.
    Recently, Ghosh-Patra \cite{ghosh2018ricci} studied $*$-Ricci soliton and gradient almost $*$-Ricci soliton on contact metric manifolds. Later on, several mathematician studied $*$-Ricci soliton on contact and almost contact metric manifolds (e.g., see \cite{dai2019non,dai2019ricci,venkatesh2019ricci,wang2020contact}). Furthermore, Venkatesh et al. \cite{venkatesh2019ricci} proved that if an $\eta$-Einstein Kenmotsu manifold admits a $*$-Ricci soliton, then it is Einstein. Recently, Wang \cite{wang2020contact} proved that if the metric of a Kenmotsu $3$-manifold represents a $*$-Ricci soliton, then the manifold is locally isometric to the hyperbolic space $\mathbb{H}^3(-1)$. Motivated by the above background, in the present paper we mainly consider almost $*$-Ricci soliton on Kenmotsu manifold.
    
    The structure of this paper is the following. After collecting some basic definition and facts on almost contact metric manifolds in Section $2$, we consider almost $*$-Ricci solitons on almost contact metric manifolds in Section $3$.
     
    \section{Preliminaries}
    Let $M^{2n+1}$ be a $(2\,n + 1)$-dimensional smooth manifold. An \textit{almost contact structure} on $M^{2n+1}$ is a triple $(\varphi,\xi,\eta)$, where $\varphi$ is a $(1,1)$-tensor, $\xi$ is a vector field (called \textit{Reeb vector field}) and $\eta$ is a 1-form, satisfying
    \begin{align}\label{2.1}
    \varphi^2 = -id + \eta\otimes \xi,\quad \eta(\xi) = 1,
    \end{align}
    where $id$ denotes the identity endomorphism. A smooth manifold $M$ with such a structure is called an \textit{almost contact manifold}. It follows from \eqref{2.1} that
    \begin{equation*}
    \varphi(\xi) = 0,\quad
    \eta \circ \varphi = 0,\quad
    {\rm rank}(\varphi)=2n.
    \end{equation*} 
    A Riemannian metric $g$ on $M$ is said to be \textit{compatible} with an almost contact structure $(\varphi,\xi,\eta)$ if
    \[
    g(\varphi X, \varphi Y) = g(X,Y) - \eta(X)\,\eta(Y)
    \]
    for any $X\in\mathfrak{X}_M$, where $\mathfrak{X}_M$ is the Lie algebra of all vector fields on $M$.
    An almost contact manifold endowed with a compatible Riemannian metric $g$ is said to be an \textit{almost contact metric manifold} and is denoted by $M(\varphi,\xi,\eta,g)$.
    The \textit{fundamental $2$-form} $\Phi$ on $M(\varphi,\xi,\eta,g)$ is defined by $\Phi(X,Y)=g(X,\varphi Y)$ for any $X\in\mathfrak{X}_M$.
    An almost contact metric manifold satisfying $d\eta = 0$ and $d\Phi = 2\,\eta\wedge\Phi$ is said to be an \textit{almost Kenmotsu manifold} (e.g.,~\cite{dileo2007almost,dileo2009almost}).
    An almost contact metric structure on $M$ is said to be \textit{normal} if
    the almost complex structure $J$ on $M^{2n+1}\times \mathbb{R}$ defined by
    $$J(X,f\frac{d}{dt})=(\varphi\,X-f\xi,\eta(X)\frac{d}{dt}),$$
    where $f$ is a real function on $M\times \mathbb{R}$, is \textit{integrable}, or the tensor $N_{\varphi}=[\varphi,\varphi] + 2d\eta\otimes \xi$ vanishes on $M$, where $[\varphi,\varphi]$ denotes the \textit{Nijenhuis tensor} of $\varphi$.
    A normal almost Kenmotsu manifold is said to be a \textit{Kenmotsu manifold} (see \cite{kenmotsu1972class}); equivalently, an almost contact metric manifold is said to be a \textit{Kenmotsu manifold} (see \cite{kenmotsu1972class}) if
    \begin{align}\label{2.2}
    (\nabla_{X}\,\varphi)Y=g(\varphi X,Y)\,\xi -\eta(Y)\,\varphi X
    \end{align}
    for any $X,\,Y\in\mathfrak{X}_M$, where $\nabla$ is the \textit{Levi-Civita connection} of $g$.
    The following formulas hold on Kenmotsu manifolds, see~\cite{ghosh2019ricci,kenmotsu1972class}:
    \begin{align}\label{2.3}
    \nabla_{X}\,\xi&= X - \eta(X)\xi,\\
    \label{2.4}
    R(X, Y)\xi&= \eta(X)Y - \eta(Y)X,\\
    \label{2.6}
     (\nabla_{\xi}Q)X&=-2\,QX -4n X,\\
    \label{2.5}
    Q\xi&= -2n\xi
    \end{align}
    for all $X,Y\in\mathfrak{X}_M$, where $R$ and $Q$ denote, respectively, the \textit{curvature tensor} and the \textit{Ricci operator} of $g$ associated with the Ricci tensor and given by
    $\Ric(X,Y)=g(QX,Y)$ for
    $X,Y\in\mathfrak{X}_M$.
    \begin{lem}\label{lem2.1}
    	(Lemma 4.2 of \cite{ghosh2020ricci}) Let $M^{2n+1}(\varphi,\xi,\eta,g)$ be a Kenmotsu manifold. If $\{e_i\}_{1\leq i \leq 2n+1}$ be an orthonormal frame on $M$ then we have for all $X\in\mathfrak{X}_M$
    	\begin{align*}
    		(i)\,\,\,\sum_{\,i=1}^{2n+1}g((\nabla_XQ)\varphi\,e_{i},e_{i})
    		&=0,\\
    	(ii)\,\,\,\sum_{\,i=1}^{2n+1}g((\nabla_{\varphi\,e_{i}}Q)\varphi X,e_{i})
    		&=-\frac{1}{2}X(r) - \{r+2n(2n+1)\}\eta(X).
    	\end{align*}
    \end{lem}
    In view of the first Bianchi's identity, the \textit{$*$-Ricci tensor} on almost contact metric manifold can be written as (see \cite{blair2010riemannian})
    \begin{align}\label{2.11A}
    \Ric^*(X,Y)=-\frac{1}{2} \sum_{i}g(R(X,\varphi Y)e_i,\varphi e_i)=\sum_{i}g(R(e_i,X)\varphi X,\varphi e_i)
    \end{align}
    for all $X$,$Y\in\mathfrak{X}_M$. Note that $\Ric^*$ is not symmetric. By contraction of \eqref{2.11A}, the \textit{$*$-scalar curvature} is given by $r^*=\sum_{i}\Ric^*(e_i,e_i)$. If the $*$-Ricci tensor $\Ric^{*}$ is a constant multiple of the Riemannian metric $g$ we say that the manifold is \textit{$*$-Einstein}. Moreover, the $*$-scalar curvature is not constant on a \textit{$*$-Einstein manifold}.
    A contact metric manifold $M^{2n+1}$ is called \textit{$\eta$-Einstein}, if its Ricci tensor has the following form
    \begin{align}\label{2.11}
    \Ric = \alpha\,g + \beta\,\eta\otimes\eta
    \end{align}
    where $\alpha$ and $\beta$ are smooth functions on $M$.
    For an $\eta$-Einstein $K$-contact manifold of dimension $>3$, these $\alpha$ and $\beta$ are constant (e.g., \cite{yano1984structures}), but for an $\eta$-Einstein Kenmotsu manifold this is not true (see details \cite{kenmotsu1972class}).
    The \textit{divergence of a tensor field} $K$ of type $(1,s)$ is defined by
    \[
    (\Div K)(X_1,X_2,\dots,X_s)=\sum\nolimits_{\,i=1}^{2n+1} g((\nabla_{e_i}K)(X_1,X_2,\dots,X_s),e_i)
    \]
    for all $X_1,X_2,\dots,X_s \in\mathfrak{X}_M$. It is well known that for an arbitrary $X\in\mathfrak{X}_M$,
    $\pounds_{X}f = X(f) = g(X,D f)$
    for some smooth function $f$, where $D$ is the gradient operator of $g$. The \textit{Hessian} of a smooth function $f$ on $M$ is defined
    for all $X,Y\in\mathfrak{X}_M$ by
    \[
    {\rm Hess}_f(X,Y) = \nabla^2f(X,Y) = g(\nabla_{X}D f,Y).
    \]
    
    \section{Main Results}
    Here we first prove the following.
    \begin{lem}\label{lem3.1}
    	Let $M^{2n+1}(\varphi,\xi,\eta,g)$ be a Kenmotsu manifold. If $g$ represents an almost $*$-Ricci soliton with the potential vector field $V$, then we have for all $X$,$Y\in\mathfrak{X}_M$
    	\begin{align}\label{3.1}
    	(\pounds_V\,R)(X,Y)\xi=&2\Big\{\eta(X)QY-\eta(Y)QX\Big\}+2\Big\{(\nabla_XQ)Y-(\nabla_YQ)X\Big\}\notag\\
    	&+ 4n\Big\{\eta(X)Y - \eta(Y)X\Big\}+ g(\xi,\nabla_{X}D\lambda)Y - g(\xi,\nabla_{Y}D\lambda)X\notag\\
    	&- \eta(Y)\,\nabla_{X}D\lambda + \eta(X)\,\nabla_{Y}D\lambda.
    	\end{align}
    \end{lem}
    \begin{proof} First we recalling a formula on Kenmotsu manifold \cite{kenmotsu1972class}:
    	\begin{align*}
    		R(X,Y)\varphi\, Z - \varphi\,R(X,Y)Z = g(Y,Z)\,\varphi X - g(X,Z)\,\varphi Y + g(X,\varphi Z)Y - g(Y,\varphi Z)X.
    	\end{align*}
     Taking its scalar product with $\varphi\,U$ and using \eqref{2.1}, \eqref{2.4}, we acquire
    	\begin{align}\label{3.2}
    	g(R(X,Y)\varphi\, Z,\varphi U) - g(R(X,Y)Z,U)&= g(Y,Z)\,g(X,U)\nonumber\\  
    	&- g(X,Z)\,g(Y,U)
    	+ g(X,\varphi Z)\,g(Y,\varphi U)\notag\\
    	&- g(Y,\varphi Z)\,g(X,\varphi U)
    	\end{align} 	
    	where we have used $\varphi$ is skew-symmetric. Considering a local orthonormal basis $\{e_i\}_{1\leq i \leq 2n+1}$ of tangent space at each point of $M$. Now substituting $e_i$ for $X$ and $Z$ in \eqref{3.2} and summing over $i:1\leq i\leq 2n+1$, we get
    	\begin{align*}
    		\Ric^*(Y,Z) = \Ric(Y,Z) + (2n-1)\,g(Y,Z) + \eta(Y)\eta(Z)
    	\end{align*}
    	for all $Z\in\mathfrak{X}_M$, where we have used \eqref{2.11A}. Then the soliton equation \eqref{1.1} can be written as
    	\begin{align}\label{3.3}
    	(\pounds_{V}\,g)(X,Y) + 2\,\Ric(X,Y) = 2\,\{\lambda -2n+1\}\,g(X,Y) - 2\,\eta(X)\eta(Y).
    	\end{align}
    	Taking its covariant derivative along $Z\in\mathfrak{X}_M$ and using \eqref{2.3}, we have
    	\begin{align}\label{3.4}
    	(\nabla_Z\,\pounds_{V}\,g)(X,Y) +2\,(\nabla_{Z}\,\Ric)(X,Y)&= 2\,Z(\lambda)\,g(X, Y)\nonumber\\
    	&-2\bigg\{\eta(X)g(Y,Z) + \eta(Y)g(X,Z)\notag\\
    	&- 2\,\eta(X)\eta(Y)\eta(Z)\bigg\}.
    	\end{align}
    	 Recall the commutation formula, [see p.~23 in \cite{yano1970integral}],
    	\begin{align}\label{3.4A}
    	(\pounds_{V}\,\nabla_{Z}\,g - \nabla_{Z}\pounds_{V}\,g - \nabla_{[V,Z]}\,g)(X,Y)=& -g((\pounds_{V}\,\nabla)(Z,X),Y)\notag\\
    	&-g((\pounds_{V}\,\nabla)(Z, Y), X).
    	\end{align}
    	Since Riemannian metric $g$ is parallel, it follows from the last equation and \eqref{3.4} that
    	\begin{align}\label{3.5}
    	\nonumber
    	g((\pounds_{V}\,\nabla)(Z,X),Y)+g((\pounds_{V}\,\nabla)(Z,Y),X)=&2\,Z(\lambda)\,g(X,Y) -2\,(\nabla_{Z}\Ric)(X,Y)\\
    	&-2\bigg\{\eta(X)g(Y,Z) + \eta(Y)g(X, Z)\notag\\
    	&- 2\,\eta(X)\eta(Y)\eta(Z)\bigg\}.
    	\end{align}
    	 Cyclically rearranging $X,\,\,Y$ and $Z$ in \eqref{3.5}, we obtain
    	\begin{align}\label{3.6}
    	\nonumber
    	g((\pounds_{V}\,\nabla)(X,Y),Z)=& (\nabla_{Z}\Ric)(X,Y)-(\nabla_{X}\Ric)(Y,Z)-(\nabla_{Y}\Ric)(Z,X)\\
    	\nonumber
    	+& X(\lambda)\,g(Y,Z) - Y(\lambda)\,g(X,Z) - Z(\lambda)\,g(X,Y)\\
    	-& 2\bigg\{\eta(Z)g(X, Y)-\eta(X)\eta(Y)\eta(Z)\bigg\}.
    	\end{align}
    	 In view of \eqref{2.3} and \eqref{2.5} one can compute the following formula
    	\begin{align}\label{3.7}
    	(\nabla_{X}Q)\xi=-QX -2n X
    	\end{align}
    	for all $X\in\mathfrak{X}_M$. Next substituting $Y$ by $\xi$ in \eqref{3.6} and using \eqref{2.6} and \eqref{3.7}, we derive that
    	\begin{align}\label{3.8}
    	(\pounds_{V}\,\nabla)(X,\xi) = 2\,QX + 4n\,X + \xi(\lambda) X + X(\lambda)\xi - \eta(X)\,D\lambda.
    	\end{align}
    	 Taking its covariant derivative of along $Y\in\mathfrak{X}_M$ and using \eqref{2.3}, we find
    	\begin{align}\label{3.9}
    	(\nabla_Y \pounds_V\,\nabla)(X, \xi) + (\pounds_V\,\nabla)(X, Y)&-\eta(Y)(\pounds_V\,\nabla)(X, \xi)\notag\\
    	&=2\,(\nabla_YQ)X + Y(\xi(\lambda))X
    	+ g(X,\nabla_{Y}D\lambda)\xi\notag\\
    	&- X(\lambda)\,\varphi^2Y -\eta(X)\,\nabla_{Y}D\lambda - g(X,Y)\,D\lambda\notag\\
    	&+ \eta(X)\eta(Y)\,D\lambda.
    	\end{align}
    	 Recalling the following commutation formula (see \cite{yano1970integral}, p.~23):
    	\begin{align}\label{3.9A}
    	(\pounds_{V}\,R)(X,Y)Z = (\nabla_{X}\pounds_{V}\,\nabla)(Y,Z) -(\nabla_{Y}\pounds_{V}\,\nabla)(X,Z).
    	\end{align}
    	Applying \eqref{3.9} in the previous formula, we obtain
    	\begin{align}\label{3.10}
    	(\pounds_V\,R)(X,Y)\xi=& \eta(X)(\pounds_V\,\nabla)(Y, \xi) - \eta(Y)(\pounds_V\,\nabla)(X, \xi)\nonumber\\
    	&+2\,\{(\nabla_XQ)Y - (\nabla_YQ)X\}
    	+ X(\xi(\lambda))Y - Y(\xi(\lambda))X \nonumber\\
    	&+ X(\lambda)\,\varphi^2Y - Y(\lambda)\,\varphi^2X
    	+\eta(X)\,\nabla_{Y}D\lambda -\eta(Y)\,\nabla_{X}D\lambda,
    	\end{align}
    	where we have used that $\pounds_V\,\nabla$ and Hess$_{\lambda}$ are symmetric. In view of \eqref{2.3}, it is easily seen that $X(\xi(\lambda)) = X(\lambda) -\xi(\lambda)\eta(X) + g(\xi,\nabla_{X}D\lambda)$. Substituting this into \eqref{3.10}, with the aid of \eqref{3.9}, we arrive at the desired result.
    \end{proof}

    \begin{lem}\label{lem3.2}
    	Let $M^{2n+1}(\varphi,\xi,\eta,g)$, $n>1$, be a Kenmotsu manifold. If $g$ represents an almost $*$-Ricci soliton with the potential vector field $V$ then we have 
    	\begin{itemize}
    	    \item [(i)] $\nabla_{\xi}D\lambda = \xi(\xi(\lambda))\,\xi$
    		\item[(ii)] $\nabla_{X}D\lambda = (\xi(\xi(\lambda)))-2\lambda)\, \varphi^2X + \xi(\xi(\lambda))\,\eta(X)\xi$
    	\end{itemize}
    	for all $X\in\mathfrak{X}_M$.
    \end{lem}
    \begin{proof}
    From \eqref{1.1} and \eqref{2.5}, it follows that $(\pounds_V\,g)(X,\xi) = 2\lambda\,\eta(X)$. Using this in the co-variant derivative of \eqref{2.4} along $V$, we get
    	\begin{align}\label{3.12}
    	(\pounds_{V}\, R)(X,Y)\xi + R(X,Y)\pounds_V\xi=&2\lambda\Big\{\eta(X)Y - \eta(Y)X\Big\}\notag\\
    	&+ g(X,\pounds_V\xi)Y-g(Y,\pounds_V\xi)X
    	\end{align}
    	for all $X,Y\in\mathfrak{X}_M$. Proceeding, we invoke Lemma \ref{lem3.1} to infer
    	\begin{align}\label{3.13}
    	2(\lambda - 2n)\,\{\eta(X)Y&-\eta(Y)X\} 
    	+ g(X,\pounds_V\xi)Y - g(Y,\pounds_V\xi)X\notag\\
    	=&2\Big\{(\nabla_XQ)Y - (\nabla_YQ)X\Big\} + 2\Big\{\eta(X)\,QY - \eta(Y)\,QX\Big\}\notag\\
    	&+ g(\xi,\nabla_{X}D\lambda)Y- g(\xi,\nabla_{Y}D\lambda)X\notag\\
    	&- \eta(Y)\,\nabla_{X}D\lambda + \eta(X)\,\nabla_{Y}D\lambda + R(X,Y)\pounds_V\xi.
    	\end{align}
    	 At this point, contracting of second Bianchi's identity we derive the formula: $\tr\{Y\longrightarrow (\nabla_YQ)X\}=\frac{1}{2}\,X(r)$. Applying this in the contraction of \eqref{3.13} over $Y$, we compute
    	\begin{align}\label{3.14}
    	X(r)+(2n-1)g(\xi, \nabla_{X}D\lambda)&+\eta(X)\, \Div D\lambda-\Ric(X,\pounds_V\xi)\notag\\
    	&= 2\{2n(\lambda-2n-1)-r\}\,\eta(X) + 2n\,g(X,\pounds_V\xi).
    	\end{align}
    	 Inserting $X=\varphi\,X$ and $Y=\varphi\,Y$ in \eqref{3.13} and recalling the formula (e.g., \cite{kenmotsu1972class}):
    	\begin{align}\label{3.15}
    	R(\varphi X,\varphi Y)Z = &R(X,Y)Z + g(Y,Z)X - g(X,Z)Y\notag\\
    	&+ g(Y,\varphi Z)\,\varphi X - g(X,\varphi Z)\,\varphi Y
    	\end{align}
    	We deduce that
    	\begin{align}\label{3.16}
    	g(X,\pounds_V\xi)Y - g(Y,\pounds_V\xi)X=&2\Big\{(\nabla_{\varphi X}Q)\,\varphi Y - (\nabla_{\varphi Y}Q)\,\varphi X\Big\}\notag\\
    	&+ g(\xi,\nabla_{\varphi X}D\lambda)\,\varphi Y - g(\xi,\nabla_{\varphi Y}D\lambda)\,\varphi X\notag\\
    	&+ R(X,Y)\pounds_V\xi. 
    	\end{align}
    	 In view of $\xi(\xi(\lambda)) = g(\xi,\nabla_{\xi}D\lambda)$, $\nabla_{\xi}\xi=0$ (follows from \eqref{2.3}) and \eqref{2.1}, one can compute $g(\varphi^2\,X,\nabla_{\xi}D\lambda) = \xi(\xi(\lambda))\,\eta(X) - g(\xi,\nabla_XD\lambda)$. Thus, at this point, contracting \eqref{3.16} over $Y$ and using the Lemma \ref{lem2.1}, we achieve
    	\begin{align}\label{3.17}
    	2n\,g(X,\pounds_V\xi)=&X(r) + 2\{r+2n(2n+1)\}\eta(X) + g(\xi,\nabla_{X}D\lambda)\notag\\
    	&- \xi(\xi(\lambda))\eta(X) - \Ric(X,\pounds_V\xi).
    	\end{align}
    	 Taking into account \eqref{3.14}, as well as \eqref{3.17}, it suffices to show that
    	\begin{align}\label{3.18}
    	2(n-1) g(\xi,\nabla_XD\lambda) =\Big\{4n\lambda - \xi(\xi(\lambda)) - \Div D\lambda\Big\}\eta(X).
    	\end{align}
    	 Noting that from \eqref{3.18}, we get
    	$\Div (D\lambda)=4n\,\lambda-(2n-1)\,\xi(\xi(\lambda))$, this reduces \eqref{3.18} to
    	\begin{align}\label{3.19}
    	(n-1)\Big\{g(\xi,\nabla_XD\lambda) - \xi(\xi(\lambda))\,\eta(X)\Big\}=0,
    	\end{align}
    	which establishes the formula $(i)$. In addition, it follows from \eqref{2.4} that 
    	$R(X,\xi)\pounds_V\xi = g(X,\pounds_V\xi)\xi - g(\xi,\pounds_V\xi)X$. Thus, replacing $Y$ by $\xi$ in \eqref{3.13} and applying \eqref{2.1}, \eqref{2.6}, \eqref{2.5} and \eqref{3.7}, we obtain
    	\begin{align}\label{3.20}
    	2\lambda\,\varphi^2X=g(\xi,\nabla_XD\lambda)\xi - g(\xi,\nabla_{\xi}D\lambda)X - \nabla_XD\lambda +\eta(X)\, \nabla_{\xi}D\lambda
    	\end{align}
    	for all $X\in\mathfrak{X}_M$. As ${\rm Hess}_{\lambda}$ is symmetric, it suffices to combine the first result of Lemma \ref{lem3.1} and \eqref{3.20} to arrive at the desired result.
    \end{proof}
    
    \begin{theo}\label{thm3.2}
    	Let $M^{2n+1}(\varphi, \xi, \eta, g)$, $n>1$, be a Kenmotsu manifold. If $g$ represents an almost $*$-Ricci soliton with the potential vector field $V$ is a Jacobi along $\xi$, then it is a steady $*$-Ricci soliton.
    \end{theo}
    \begin{proof}
    	First of all, invoking Lemma \ref{lem3.2}, we derive the following expression
    	\begin{align}\label{3.24}
    	R(X, Y)D\lambda=&2\{Y(\lambda)-Y(\xi(\xi(\lambda)))\}\,\varphi^2X - \Big\{X(\lambda)-X(\xi(\xi(\lambda)))\Big\}\,\varphi^2Y\notag\\
    	&- 2(\lambda-\xi(\xi(\lambda)))\Big\{\eta(Y)X-\eta(X)Y\Big\}\notag\\
    	&+ X(\xi(\xi(\lambda)))\,\eta(Y)\xi -Y(\xi(\xi(\lambda)))\,\eta(X)\xi
    	\end{align}
    	for all $X,Y\in\mathfrak{X}_M$. In view of \eqref{2.1}, \eqref{2.3} and \eqref{2.5}, last equation provides $X(\xi(\xi(\lambda))) = \xi(\xi(\xi(\lambda)))\,\eta(X)$. Applying this in \eqref{3.24}, we have
    	\begin{align}\label{3.25}
    	R(X, Y)D\lambda=& 2\Big\{X(\lambda)Y - Y(\lambda)X\Big\} - 2\Big\{X(\lambda)\eta(Y)\xi - Y(\lambda)\eta(X)\xi\Big\}\nonumber\\
    	+&\Big\{\xi(\xi(\xi(\lambda))) - 2(\lambda-\xi(\xi(\lambda)))\Big\}\big\{\eta(Y)X - \eta(X)Y\big\}.
    	\end{align}
    	 Contracting this over $X$, we infer
    	\begin{align}\label{3.26}
    	\Ric(Y,D\lambda)=&-2(2n+1)Y(\lambda)+2\,\xi(\lambda)\,\eta(Y)\notag\\
    	&+ 2n\Big\{\xi(\xi(\xi(\lambda)))-2(\lambda-\xi(\xi(\lambda)))\Big\}\,\eta(Y).
    	\end{align}
    	 On the other hand, inserting $X=\varphi X$ and $Y=\varphi Y$ in \eqref{3.24} and using the well known formula, see \cite{kenmotsu1972class}:
    	\begin{align*}
    		R(\varphi X,\varphi Y)Z = &R(X, Y)Z + g(Y,Z)X - g(X,Z)Y\notag\\
    		&+ g(Y,\varphi Z)\varphi X - g(X,\varphi Z)\varphi Y.
    	\end{align*}
    	We acquire
    	\begin{align}\label{3.27}
    	R(X, Y)D\lambda = X(\lambda)Y - Y(\lambda)X +\Big\{g(\varphi X,D\lambda)\,\varphi Y - g(\varphi Y,D\lambda)\,\varphi X\Big\}. 
    	\end{align}
    	 Contracting this over $X$ and recalling \eqref{3.25}, we achieve
    	\begin{align}\label{3.28}
    	(2n+1)Y(\lambda) - \xi(\lambda)\eta(Y) = 2n\Big\{\xi(\xi(\xi(\lambda)))-2(\lambda-\xi(\xi(\lambda)))\Big\}\eta(Y).
    	\end{align}
    	 At this point, replacing $Y$ by $\varphi Y$ in \eqref{3.28}, we get $g(\varphi Y,D\lambda)=0$ for all $Y\in\mathfrak{X}_M$, consequently, 
    	\begin{align}\label{3.28A}
    	Y(\lambda) = \xi(\lambda)\,\eta(Y).
    	\end{align}
    	 Now recalling the well known formula (see p.~23 of \cite{yano1970integral}):
    	\begin{align}\label{3.29}
    	(\pounds_{V}\,\nabla)(X, Y)=\nabla_{X}\nabla_{Y}V - \nabla_{\nabla_{X}Y}V - R(X,V)Y.
    	\end{align}
    	 By hypothesis: $V$ is a Jacobi along $\xi$, i.e., 
    	$\nabla_{\xi}\nabla_{\xi}V = R(V,\xi)\xi = 0$. Thus, it suffices to combine \eqref{3.8} and \eqref{3.29} to arrive $D\lambda=2\,\xi(\lambda)\xi$, which yields from \eqref{3.28A} that $\xi(\lambda)=0$, consequently, $D\lambda=0$. Hence $\lambda$ is a constant. Further, from \eqref{3.28A}, we get $\varphi D\lambda=0$. Its covariant derivative along $X\in\mathfrak{X}_M$. Now together with \eqref{2.1}, \eqref{2.2} and Lemma \ref{lem3.2} gives that 
    	\begin{align*}
    	    g(\varphi X,D\lambda)\xi +2\lambda\,\varphi X = \{\xi(\lambda)+\xi(\xi(\lambda))\} \,\varphi X
    	\end{align*} 
    	which settles our claim.
    \end{proof}
    
    Next we prove the following result by  considering Kenmotsu metric $g$ as an almost $*$-Ricci soliton.
    \begin{theo}\label{thm3.1}
    	Let $M^{2n+1}(\varphi, \xi, \eta, g)$, $n>1$, be an $\eta$-Einstein Kenmotsu manifold. If $g$ represents an almost $*$-Ricci soliton with the potential vector field $V$, then it is Einstein with Einstein constant $-2n$.
    \end{theo}
    \begin{proof} First substituting $\xi$ for $Y$ in \eqref{3.1} and recalling \eqref{2.1}, \eqref{2.6}, \eqref{2.5}, \eqref{3.7} and Lemma \ref{lem3.2}, we obtain 
    	\begin{align}\label{3.1A}
    	(\pounds_V\,R)(X,Y)\xi = \lambda\,\varphi^2X
    	\end{align}
    	for any $X\in\mathfrak{X}_M$. At this point, taking trace of \eqref{2.11}, we find $r=(2n+1)\alpha+\beta$. On the other hand, inserting $X=Y=\xi$ in \eqref{3.15} and using \eqref{2.5}, we get $\alpha+\beta=-2n$. Thus, equation \eqref{2.11} becomes
    	\begin{align}\label{3.21}
    	QX = \big(1+\frac{r}{2n}\big)\,X -\big(2n+1+\frac{r}{2n}\big)\,\eta(X)\xi.
    	\end{align}
    	In view of \eqref{2.3}, we can compute from foregoing equation that
    	\begin{align*}
    	    (\nabla_X Q)Y-(\nabla_Y Q)X =&\frac{1}{2n}\Big\{X(r)\varphi^2 Y-Y(r)\varphi^2 X\Big\}\\ &-\big(2n+1+\frac{r}{2n}\big)\Big\{\eta(X)Y-\eta(Y)X\Big\}.
    	\end{align*}
    	 Thus, combining \eqref{2.1}, \eqref{3.9}, \eqref{3.21} and Lemma \ref{lem3.2}, we deduce
    	\begin{align}\label{3.23}
    	(\pounds_V\,R)(X, Y)\xi =&\frac{1}{n}\Big\{X(r)\,\varphi^2 Y-Y(r)\,\varphi^2 X\Big\}\notag\\
    	&+\lambda\Big\{\eta(Y)\,\varphi^2X-\eta(X)\,\varphi^2Y\Big\}.
    	\end{align}
    	 Now replacing $Y$ by $\xi$ in \eqref{3.23} and noting that \eqref{3.1A}, we get $\xi(r)\,\varphi^2 X=0$ for any $X\in\mathfrak{X}_M$, consequently, $\xi(r)=0$. Again, \eqref{2.5} follows that $\xi(r)=-2(r+2n(2n+1))$, and therefore $r=-2n(2n+1)$. This shows from \eqref{3.21} that $\Ric=-2n\,g$, which finishes the proof.
    \end{proof}
    
    \begin{theo}
    	If a Kenmotsu manifold $M^{2n+1}(\varphi, \xi, \eta, g)$, $n>1$, admits an almost $*$-Ricci soliton with non-zero potential vector field $V$ is collinear to $\xi$, then it is Einstein with Einstein constant $-2n$.
    \end{theo}
    \begin{proof} By our assumption: $V=\gamma\,\xi$ for some smooth function $\gamma$ on $M$. Its covariant derivative along $X\in\mathfrak{X}_M$ gives
    	$\nabla_X V = X(\gamma)\xi +\gamma (X-\eta(X)\xi)$, where we have used \eqref{2.3}. In view of this, Eq. \eqref{1.1} becomes
    	\begin{align}\label{3.23A}
    	2\,\Ric(X,Y)+X(\gamma)\,\eta(Y)+Y(\gamma)\,\eta(X) =&2(\lambda-\gamma-2n+1)\, g(X,Y)\notag\\
    	&+2(\gamma-1)\,\eta(X)\eta(Y)
    	\end{align}
    	for all $X,Y\in\mathfrak{X}_M$. Replacing both $X$, $Y$ by $\xi$ in \eqref{3.23A} and recalling \eqref{2.5}, we acquire $\xi(\gamma)=\lambda$.
    	It follows from \eqref{3.23A} that $X(\gamma)=\lambda\,\eta(X)$. The Eq. \eqref{3.23A} reduces to following
    	\begin{align}\label{3.24A}
    	\Ric= (\lambda-\gamma-2n+1)\,g - (\lambda-\gamma+1)\,\eta\otimes\eta.
    	\end{align}
    	Consequently, $g$ is $\eta$-Einstein. Invoking Theorem~\ref{thm3.1} we conclude that $g$ is Einstein. It follows from \eqref{3.24A} that $\gamma=\lambda+1$, and therefore,  Eq. \eqref{3.24A} implies that $\Ric=-2ng$. So, the proof is completed.
    \end{proof}
    
    A vector field $X$ on a contact metric manifold $M$ is called a \textit{contact or infinitesimal contact transformation} if it preserves the contact form $\eta$, i.e., there is a smooth function $f:M\to\mathbb{R}$ satisfying
    \begin{align}\label{3.30}
    \pounds_{X}\,\eta=f\,\eta,
    \end{align}
    and if $f=0$ then the vector field $X$ is called \textit{strict}. Now we consider a Kenmotsu metric as an almost $*$-Ricci soliton with the potential vector field $V$ is contact or infinitesimal contact transformation and prove the following.
    \begin{theo}\label{thm3.4}
    	If a Kenmotsu manifold $M^{2n+1}(\varphi, \xi, \eta, g)$ admits an almost $*$-Ricci soliton with the potential vector field $V$ is contact or infinitesimal contact transformation, then the metric $g$ is Einstein metric with Einstein constant $-2n$. Moreover, $V$ is strict.
    \end{theo}
    \begin{proof}
    	From \eqref{3.3} and \eqref{2.5}, we get $(\pounds_{V}\,g)(X,\xi) = 2\lambda\,\eta(X)$. In addition, it follows from $\eta(X)=g(X,\xi)$ that
    	\begin{align}\label{3.31}
    	(\pounds_{V}\,\eta)(X) = 2\lambda\,\eta(X) + g(X,\pounds_V\xi)
    	\end{align}
    	for all $X\in\mathfrak{X}_M$. It suffices to use \eqref{3.30} to infer $g(X,\pounds_V\xi) = (f-2\lambda)\,\eta(X)$. Easily one verifies that $(\pounds_{V}\,\eta)(\xi) + \eta(\pounds_V\xi)=0$. Combining this with \eqref{3.31}, we deduce $\eta(\pounds_V\xi)=-\lambda$, consequently, $f=\lambda$. Thus, $\pounds_V\xi=-\lambda\,\xi=[V,\xi]$. Taking its covariant derivative along $X\in\mathfrak{X}_M$ and recalling \eqref{2.3} gives $\nabla_X[V,\xi]=-X(\lambda)\xi - \lambda(X-\eta(X)\xi)$. From this fact, it is easy to see that
    	\begin{align}\label{3.32}
    	(\pounds_{[V,\xi]}\,g)(X,Y) =&-X(\lambda)\,\eta(Y) -Y(\lambda)\,\eta(X)\\
    	&-2\lambda\Big\{g(X,Y) - \eta(X)\eta(Y)\Big\}.
    	\end{align}
    	 On the other hand, Lie-derivative of \eqref{3.3} along $\xi$ provides
    	\begin{align}\label{3.33}
    	(\pounds_{\xi}\,\pounds_{V}\,g)(X,Y)&+2(\pounds_{\xi}\Ric)(X,Y)\notag\\
    	=&2\xi(\lambda)\,g(X,Y) + 2(\lambda - 2n + 1) (\pounds_{\xi}\,g)(X,Y) \nonumber\\ 
    	&- 2\Big\{\eta(Y)\,(\pounds_{\xi}\,\eta)(X) + \eta(X)\,(\pounds_{\xi}\,\eta)(Y)\Big\}.
    	\end{align}
    	According to \eqref{2.3}, it is easily seen that
    	\begin{align}\label{3.34}
    	(\pounds_{\xi}\,g)(X,Y) = 2\Big\{g(X,Y) - \eta(X)\eta(Y)\Big\}
    	\end{align}
    	Taking the covariant derivative along $Z\in\mathfrak{X}_M$  and using \eqref{2.3} yields
    	\begin{align}\label{3.35}
    	(\nabla_Z\pounds_{\xi}\,g)(X, Y) = 2\Big\{2\eta(X)\eta(Y)\eta(Z) - g(X, Z)\eta(Y) - g(Y, Z)\eta(X)\Big\}.
    	\end{align}
    	 Combining this with \eqref{3.4A}, we acquire
    	\begin{align*}
    		g((\pounds_{\xi}\,\nabla)(Z,X),Y)&-g((\pounds_{\xi}\,\nabla)(Z,Y),X)\\
    		&= 2\Big\{2\eta(X)\eta(Y)\eta(Z) - g(X,Z)\eta(Y) - g(Y,Z)\eta(X)\Big\},
    	\end{align*}
    	where we have used the Riemannian metric $g$ is parallel. By direct calculation, last equation provides 
    	\begin{align}\label{3.36}
    	(\pounds_{\xi}\nabla)(X,Y) = -2\Big\{g(X,Y) - \eta(X)\eta(Y)\Big\}.
    	\end{align}
    	 Differentiating this covariantly along $Z\in\mathfrak{X}_M$ and recalling \eqref{2.3}, we compute
    	\begin{align*}
    		(\nabla_Z\pounds_{\xi}\nabla)(X,Y) &=2\bigg\{g(X,Z)\,\eta(Y)\xi + g(Y,Z)\,\eta(X)\xi + g(X,Y)\,\eta(Z)\xi \nonumber\\
    		-&g(X,Y)Z + \eta(X)\eta(Y)Z - 3\,\eta(X)\eta(Y)\eta(Z)\xi\bigg\}.
    	\end{align*}
    	Applying this in the formula \eqref{3.9A}, we derive that
    	\begin{align}\label{3.37}
    	(\pounds_{\xi}R)(Z,X)Y = 2\Big\{g(Y,Z)X - g(X,Y)Z + \eta(X)\eta(Y)Z - \eta(Y)\eta(Z)X\Big\}.	
    	\end{align}
    	Contracting \eqref{3.37} over $Z$, we achieve
    	\begin{align}\label{3.38}
    	(\pounds_{\xi}\Ric)(X,Y) = -4n\Big\{g(X,Y) - \eta(X)\eta(Y)\Big\}.
    	\end{align}
    	In view of \eqref{3.34} and \eqref{3.38}, Eq. \eqref{3.33} transform into
    	\begin{align}\label{3.39}
    	(\pounds_{\xi}\,\pounds_{V}\,g)(X,Y) = 2\Big\{2(\lambda+1) + \xi(\lambda)\Big\}g(X,Y) - 4(\lambda + 1)\,\eta(X)\eta(Y).
    	\end{align}
    	Recalling $\pounds_V\xi = -\lambda\,\xi$ and \eqref{3.31}, we get $(\pounds_V\eta)(X)=\lambda\,\eta(X)$. Using this and \eqref{3.3} in the Lie-derivative of \eqref{3.34} along the potential vector field $V,$ we arrive at
    	\begin{align}\label{3.40}
    	(\pounds_{V}\,\pounds_{\xi}\, g)(X,Y) = -4\Big\{&\Ric(X,Y) - (\lambda - 2n +1)g(X,Y)\notag\\
    	&+ (\lambda + 1)\,\eta(X)\eta(Y)\Big\}.
    	\end{align}
    	In particular, the formula (see \cite{kenmotsu1972class}): 
    	$\pounds_{X}\,\pounds_{Y}\,g - \pounds_{Y}\,\pounds_{X}\,g = \pounds_{[X,Y]}\,g$ can be written in the following form
    	\begin{align*}
    	    \pounds_{V}\,\pounds_{\xi}\,g(X,Y) - \pounds_{\xi}\,\pounds_{V}\,g(X,Y) = \pounds_{[V,\xi]}\,g(X,Y)
    	\end{align*}
    	for all $X,Y\in\mathfrak{X}_M$. Now it suffices to combine \eqref{3.32}, \eqref{3.39} and \eqref{3.40} in previous equation to arrive 
    	\begin{align}\label{3.41}
    	4\,\Ric(X,Y) - &X(\lambda)\,\eta(Y) -Y(\lambda)\,\eta(X)\notag\\
    	&= 2(\lambda - 4n - \xi(\lambda))\,g(X,Y) -10 \lambda\,\eta(X)\eta(Y).
    	\end{align}
    	At this point, replacing $\xi$ for both $X$, $Y$ in \eqref{3.41} and using \eqref{2.5}, we achieve $\lambda=0$, consequently, $f=0$, and therefore, $V$ is strict. Further, Eq. \eqref{3.41} yields that $\Ric =-2n\,g$, as required.
    \end{proof}
     \section*{Acknowledgements}
     The authors thank the referees for their valuable and constructive comments
     for modifying the presentation of this work. This research was funded by the Deanship of Scientific Research at Princess Nourah bint Abdulrahman University through the Fast-track Research Funding Program. The author was financially supported by Council for Higher Education Planning and Budgeting Committee (PBC) through the University of Haifa, Israel.


\begin{thebibliography}{99}
    	
    	
    	
    	
    	\bibitem{barros2012some}
    	Barros, A.; Ribeiro, E., Jr. {\it Some characterizations for compact almost Ricci solitons.} Proc. Amer. Math. Soc. \textbf{140}(3), (2012); 1033--1040. 
    	
    	
    	
    	\bibitem{blair2010riemannian}
    	Blair, D. E., \textit{Riemannian geometry of contact and symplectic manifolds.} Birkh\"{a}user, Boston, (2002).
    	
    	%
    	%
    	
    	\bibitem{xiaomin2018real} Chen, X., {\it Real hypersurfaces with $^*$-Ricci solitons of non-flat complex space forms.} Tokyo J. Math. \textbf{41} (2018), 433--451.
    	\bibitem{Chen1} Chen, X., {\it Quasi-Einstein structures and almost cosymplectic manifolds.} Rev. R. Acad. Cienc. Exactas Fís. Nat. Ser. A Mat. RACSAM. \textbf{114} (2020), 72, 14
    	\bibitem{Chen2} Chen, X., Cui, X., {\it The k-almost Yamabe solitons and contact metric manifolds.}
    	Rocky Mountain J. Math. (2020).
    	
    	\bibitem{dai2019non} Dai, X.: {\it Non-existence of $*$-Ricci solitons on $(\kappa, \mu)$-almost cosymplectic manifolds.} J. Geom. \textbf{110}(2), (2019); 30.  
    	
    	\bibitem{dai2019ricci} 
    	Dai, X., Zhao, Y. and De, U. C., {\it $*$-Ricci soliton on ($\kappa, \mu)$-almost Kenmotsu manifolds.} Open Math. \textbf{17}(2019); 874--882. 
    	
    	
    	\bibitem{dileo2009almost}
    	Dileo, G. and Pastore, A. M., {\it Almost Kenmotsu manifolds and nullity distributions.} J. Geom. \textbf{93}, (2009); 46--61. 
    	
    	\bibitem{dileo2007almost}
    	Dileo, G. and Pastore, A. M., {\it Almost Kenmotsu manifolds and local symmetry.} Bull. Belg. Math. Soc. Simon Stevin \textbf{14}(2), (2007); 343--354. 
    	
    	
    	\bibitem{fernandez2008remark}
    	Fern{\'a}ndez-L{\'o}pez, M. and Garc{\'\i}a-R{\'\i}o, E., {\it A Remark on compact Ricci solitons.} Math. Ann. \textbf{340}, (2008); 893--896. 
    	
    	\bibitem{ghosh2013eta}
    	Ghosh, A., {\it An $\eta$-Einstein Kenmotsu metric as a Ricci soliton.} Publ. Math. Debrecen \textbf{82}(3-4), (2013); 591--598. 
    	
    	
    	\bibitem{ghosh2020ricci} 
    	Ghosh, A., {\it Ricci almost soliton and almost Yamabe soliton on Kenmotsu manifold.} Asian Eur. J. Math. (2020), doi.org/10.1142/S1793557121501308.
    	
    	\bibitem{ghosh2019ricci}
    	Ghosh, A., {\it Ricci soliton and Ricci almost soliton within the framework of Kenmotsu manifold.} Carpathian Math. Publ. \textbf{11}(1), (2019); 59--69.
    	
    	\bibitem{ghosh2018ricci} 
    	Ghosh, A. and Patra, D. S., {\it $*$-Ricci Soliton within the frame-work of Sasakian and $(\kappa,\mu)$-contact manifold.} Int. J. Geom. Methods Mod. Phys. \textbf{15}(7), (2018); pp-1850120. 
    	
    	
    	
    	\bibitem{hamada2002real} 
    	Hamada, T., {\it Real hypersurfaces of complex space forms in terms of Ricci $*$-tensor.} Tokyo J. Math. \textbf{25}, (2002); 473--483. 
    	
    	\bibitem{hamilton1988ricci}
    	Hamilton, R.S., {\it The Ricci fow on surfaces, In: Mathematics and General Relativity.} Contemp. Math. \textbf{71}, Amer. Math. Soc. (1988); pp. 237--262. 
    	
    	\bibitem{ivey2010ricci} 
    	Ivey, T. and Ryan, P.J., {\it The $*$-Ricci tensor for hypersurfaces in $CP_n$ and $CH_n$.} Tokyo J. Math. \textbf{34}, (2011); 445-471. 
    	
    	\bibitem{kaimakamis2014ricci2} 
    	Kaimakamis, G. and Panagiotidou, K.: {\it $*$-Ricci solitons of real hypersurface in non-flat complex space forms} J. Geom. Phys., \textbf{76}, (2014); 408--413. 
    	
    	
    	\bibitem{kenmotsu1972class}
    	Kenmotsu, K., {\it A class of almost contact Riemannian manifolds.} T\^{o}hoku Math. J. \textbf{24}, (1972); 93--103. 
    	
    	
    	\bibitem{munteanu2013gradient}
    	Munteanu, O. and Sesum, N., {\it On Gradient Ricci Solitons.} J. Geom. Anal. \textbf{23}(2), (2013); 539--561 
    	
    	
    	\bibitem{P20} Patra, D.S.: {\it $K$-contact metrics as Ricci almost solitons.} Beitr. Algebra Geom.  (2020), doi.org/10.1007/s13366-020-00539-y
    	
    	\bibitem{patra2020almost} Patra, D.S. and Rovenski, V., {\it Almost $\eta$-Ricci solitons on Kenmotsu manifolds.}  (2020), arXiv preprint arXiv:2008.12497 
    	\bibitem{pigola2011ricci}
    	Pigola, S., Rigoli, M., Rimoldi, M. and Setti, A., {\it Ricci almost solitons.} Ann. Sc. Norm. Sup. Pisa Cl. Sci. \textbf{10}, 757--799 (2011)
    	
    	\bibitem{petersen2009rigidity}
    	Petersen, P. and Wylie, W., {\it Rigidity of gradient Ricci solitons.} Pacific J. Math. \textbf{241}(2), (2009); 329--345. 
    	
    	
    	
    	\bibitem{sharma2008certain}
    	Sharma, R., {\it Certain results on $K$-contact and $(\kappa,\mu)$-contact manifolds.} J. Geom. \textbf{89}, (2008); 138--147. 
    	
    	\bibitem{sharma2014almost}
    	Sharma, R., {\it Almost Ricci solitons and $K$-contact geometry.} Monatsh. Math. \textbf{175}(4), (2015) 621--628. 
    	
    	\bibitem{tachibana1959} 
    	Tachibana, S.: On almost-analytic vectors in almost K$\ddot{a}$hlerian manifolds, Tohoku Math. J. \textbf{11}, 247--265 (1959)
    	
    	\bibitem{venkatesh2019ricci} 
    	Venkatesh, V., Naik, D.M. and Kumara, H.A., {\it $*$-Ricci solitons and gradient almost $*$-Ricci solitons on Kenmotsu manifolds.} Math. Slovaca \textbf{69}(6), (2019); 1447--1458. 
    	
    	\bibitem{wang2020contact} 
    	Wang Y., {\it Contact $3$-manifolds and $*$-Ricci soliton.} Kodai Math. J., \textbf{43}(2), (2020); 256--267. 
    	
    	\bibitem{wang2017ricci}
    	Wang Y., {\it Ricci solitons on $3$-dimensional cosymplectic manifold.} Math. Slovaca \textbf{67}(4), (2017); 979--984. 
    	
    	\bibitem{yano1970integral}
    	Yano, K., \textit{Integral formulas in Riemannian geometry.} Vol. \textbf{1}. M. Dekker. (1970)
    	
    	\bibitem{yano1984structures}
    	Yano, K. and Kon, M., \textit{Structures on manifolds.} Series in Pure Mathematics, 3, World Scientific Pub. Co., Singapore, (1984)
    	
    \end{thebibliography}
\end{document}